\documentclass[10pt,conference]{ieeeconf}

\usepackage{mathrsfs}
\usepackage{graphics} 
\usepackage{amsmath} 
\usepackage{amssymb}  

\usepackage{amsthm}

\usepackage{cite}
\usepackage{bm}
\usepackage{acronym}
\usepackage{paralist}
\usepackage{float}
\usepackage{color}
\usepackage{epstopdf}
\usepackage{multicol}
\usepackage{tikz}
\usepackage{graphicx}
\usepackage{soul}
\usepackage{mathtools}

\makeatletter
\let\NAT@parse\undefined
\makeatother

\usepackage{hyperref}
\makeatletter
\hypersetup{colorlinks=true}
\AtBeginDocument{\@ifpackageloaded{hyperref}
  {\def\@linkcolor{blue}
  \def\@anchorcolor{red}
  \def\@citecolor{red}
  \def\@filecolor{red}
  \def\@urlcolor{red}
  \def\@menucolor{red}
  \def\@pagecolor{red}
\begingroup
  \@makeother\`%
  \@makeother\=%
  \edef\x{%
    \edef\noexpand\x{%
      \endgroup
      \noexpand\toks@{%
        \catcode 96=\noexpand\the\catcode`\noexpand\`\relax
        \catcode 61=\noexpand\the\catcode`\noexpand\=\relax
      }%
    }%
    \noexpand\x
  }%
\x
\@makeother\`
\@makeother\=
}{}}
\makeatother

\newtheorem{Theorem}{Theorem}

\newtheorem{Lemma}{Lemma}

\newtheorem{Remark}{Remark}

\newtheorem{Corollary}{Corollary}

\IEEEoverridecommandlockouts

\def\BibTeX{{\rm B\kern-.05em{\sc i\kern-.025em b}\kern-.08em
    T\kern-.1667em\lower.7ex\hbox{E}\kern-.125emX}}

\begin{document}

\title{\LARGE \bf{Characterization of Domain of Fixed-time Stability under Control Input Constraints}}

\author{Kunal Garg \and Dimitra Panagou
\thanks{
The authors would like to acknowledge the support of the Air Force Office of Scientific Research under award number FA9550-17-1-0284.}
\thanks{The authors are with the Department of Aerospace Engineering, University of Michigan, Ann Arbor, MI, USA; \texttt{\{kgarg, dpanagou\}@umich.edu}.}
}
\maketitle

\begin{abstract}
In this paper, we study the effect of control input constraints on the domain of attraction of an FxTS equilibrium point. We first present a new result on FxTS, where we allow a positive term in the time derivative of the Lyapunov function. We provide analytical expressions for the domain of attraction and the settling time to the equilibrium in terms of the coefficients of the positive and negative terms that appear in the time derivative of the Lyapunov function. We show that this result serves as a robustness characterization of FxTS equilibria in the presence of an additive, vanishing disturbances. We use the new FxTS result in formulating a provably feasible quadratic program (QP) that computes control inputs that drive the trajectories of a class of nonlinear, control-affine systems to a goal set, in the presence of control input constraints. 
\end{abstract}

\section{Introduction}
In control problems where the objective is to stabilize closed-loop trajectories to a given desired point or a set, control Lyapunov functions (CLFs) are very commonly used to design the control input \cite{romdlony2016stabilization,ames2014rapidly}. Traditionally, CLFs have been used to design closed-form expressions for control inputs using Sontag's formula \cite{sontag1989universal,romdlony2016stabilization}. More recently, quadratic programs (QPs) have gained popularity for control synthesis; with this approach, the CLF conditions are formulated as inequalities that are linear in the control input \cite{li2018formally,ames2017control}, and the control input is computed as a solution to these parametric QPs. These methods are suitable for real-time implementation as QPs can be solved very efficiently. In most of the prior work on QP-based control design, the feasibility of the underlying QP is not guaranteed particularly in the presence of input constraints.

The work in \cite{li2018formally,ames2014rapidly,ames2017control} considers the design of control laws so that reachability objectives, such as reaching the desired goal set, are achieved as time goes to infinity, i.e., asymptotically or exponentially. Much attention has been paid recently to the concepts of finite- and fixed-time stability, where the system trajectories reach an equilibrium point or a set in a \textit{finite} or \textit{fixed} time, as opposed to asymptotically or exponentially. Fixed-time stability (FxTS), introduced in \cite{polyakov2012nonlinear}, is a stronger notion than exponential stability, where the time of convergence is finite, and is uniformly bounded for all initial conditions. Research has also shown that there is also a correlation between a faster rate of convergence and better disturbance rejection properties for a dynamical system \cite{bhat2000finite,polyakov2012nonlinear}. A lot of work has been since done in the field of FxTS; the authors in \cite{lopez2019conditions,lopez2018necessary} discuss necessary and sufficient conditions for FxTS; \cite{li2017fixed,corradini2018nonsingular} present FxTS results from a sliding-mode perspective (see also \cite{lopez2018fixed,liu2018distributed,wei2018observer} for some examples of applications of FxTS theory in control and estimation problems). Recently, the concept of fixed-time CLF (FxT-CLF) was introduced \cite{garg2019control}, which combines the notion of CLF and FxTS in a QP, but without any feasibility guarantees. 

The aforementioned papers study \textit{global} FxTS, which requires unbounded control authority. Since it is not possible to guarantee FxTS from arbitrary initial conditions in the presence of control input constraints, it is important to study the domain of attraction from which FxTS can be guaranteed in the presence of input bounds. To this end, in this paper, we present new Lyapunov conditions on FxTS by introducing a (possibly positive) linear term in the upper bound of the derivative of the Lyapunov function.
We show that FxTS can still be guaranteed from a domain of attraction that depends upon the relative magnitude of the positive and the negative terms in the bound of the time derivative of the Lyapunov function. We compute an upper bound on the time of convergence to the equilibrium, which is also a function of the relative magnitude of the positive and negative terms. We discuss the relation between the proposed results on FxTS and the robustness of FxTS systems under additive vanishing 
disturbances. Besides, based on the results in \cite{garg2019prescribedTAC}, we use the new FxTS conditions in a QP formulation, where the control objective is to drive closed-loop trajectories to a goal set in a given fixed time, in the presence of control input constraints. The results of this paper extend and formalize the results in \cite{garg2019control} in a QP framework, such that feasibility, as well as fixed-time convergence, can be simultaneously guaranteed from a domain of attraction that is a function of the input bounds and time of convergence. We perform numerical experiments to relate the domain of attraction with the required time of convergence and with the control input bounds. 

\section{Mathematical Preliminaries}\label{sec: math prelim}
\noindent\textbf{Notations}: In the rest of the paper, $\mathbb R$ denotes the set of real numbers, and $\mathbb R_+$ denotes the set of non-negative real numbers. We use $\|\cdot\|$ to denote the Euclidean norm. We use $\partial S$ to denote the boundary of a closed set $S$ and $\textrm{int}(S) = S\setminus \partial S$, to denote its interior. 

Next, we review the notion of fixed-time stability. Consider the nonlinear system
\begin{align}\label{ex sys}
\dot x(t) = f(x(t)), \quad x(0) = x_0,
\end{align}
where $x\in \mathbb R^n$ and $f: \mathbb R^n \rightarrow \mathbb R^n$ is continuous with $f(0)=0$. Assume that the solution of \eqref{ex sys} exists and is unique. 
The authors in \cite{polyakov2012nonlinear} presented the following result for FxTS.

\begin{Lemma}[\hspace{-0.5pt}\cite{polyakov2012nonlinear}]\label{FxTS TH}
Suppose there exists a continuously differentiable, positive definite, radially unbounded function $V$ for the system \eqref{ex sys} such that 
\begin{align}\label{eq: dot V FxTS old }
    \dot V(x) \leq -aV(x)^p-bV(x)^q,
\end{align}
for all $x\neq 0$, where $a,b>0$, $0<p<1$ and $q>1$. Then, the origin of \eqref{ex sys} is FxTS, and the time of convergence $T$ is uniformly bounded as $T \leq \frac{1}{a(1-p)} + \frac{1}{b(q-1)}$.
\end{Lemma}

\section{Main results}\label{sec: main results}
In this section, we present a new result on FxTS. Particularly, we introduce another term in the upper bound of $\dot V$ in \eqref{eq: dot V FxTS old }, and allow this term to take positive values. Consider a positive definite, continuously differentiable function $V:\mathbb R^n\rightarrow\mathbb R$, such that its time derivative along the trajectories of \eqref{ex sys} satisfies
\begin{align}
    \hspace{-5pt}\dot V(x(t)) \leq -\alpha_1V(x(t))^{\gamma_1}-\alpha_2V(x(t))^{\gamma_2}+\delta_1V(x(t)),
\end{align}
for all $t\geq 0$, with $\alpha_1, \alpha_2>0$, $\delta_1\in \mathbb R$, $\gamma_1 = 1+\frac{1}{\mu}$, $\gamma_2 = 1-\frac{1}{\mu}$ for some $\mu>1$. 

\vspace{2pt}
\noindent \textbf{New FxTS Lyapunov conditions}: Before presenting the first main result, we need the following lemma. 

\begin{Lemma}\label{lemma:int dot V}
Let $V_0, \alpha_1, \alpha_2>0$, $\delta_1\in \mathbb R$, $\gamma_1 = 1+\frac{1}{\mu}$ and $\gamma_2 = 1-\frac{1}{\mu}$, where $\mu>1$. Define 
\begin{align}\label{eq:int dot V}
    I \coloneqq \int_{V_0}^{0}\frac{dV}{-\alpha_1V^{\gamma_1}-\alpha_2V^{\gamma_2}+\delta_1V}.
\end{align}
Then, the following holds:
\begin{itemize}
    \item[(i)] If $0\leq \delta_1<2\sqrt{\alpha_1\alpha_2}$, we have for all $V_0\geq 0$
    \begin{align}\label{eq: I bound 1}
        I\leq \frac{\mu}{\alpha_1k_1}\left(\frac{\pi}{2}-\tan^{-1}k_2\right),
    \end{align}
    where $k_1 = \sqrt{\frac{4\alpha_1\alpha_2-\delta_1^2}{4\alpha_1^2}}$ and $k_2 = -\frac{\delta_1}{\sqrt{4\alpha_1\alpha_2-\delta_1^2}}$;
    \item[(ii)] If $\delta_1\geq 2\sqrt{\alpha_1\alpha_2}$ and $V_0^\frac{1}{\mu}\leq k\frac{\delta_1-\sqrt{\delta_1^2-4\alpha_1\alpha_2}}{2\alpha_1}$ with $0<k<1$, we have for all $V_0\geq 0$
    \begin{align}\label{eq: I bound 2}
        I \leq \frac{\mu k}{(1-k)\sqrt{\alpha_1\alpha_2}}.
    \end{align}
\end{itemize}
\end{Lemma}

\noindent Lemma \ref{lemma:int dot V} gives upper bounds on the integral $I$ for various cases (which will serve as the upper-bound on the fixed time of convergence as shown next). The proof is provided in Appendix \ref{app Lemma int dot V proof}. Now we are ready to present our main result on new Lyapunov conditions for FxTS. 

\begin{Theorem}\label{Th: FxTS new}
Let $V:\mathbb R^n\rightarrow \mathbb R$ be a continuously differentiable, positive definite, proper function, satisfying
\begin{align}\label{eq: dot V new ineq}
     \dot V(x) \leq     -\alpha_1V(x)^{\gamma_1}-\alpha_2V(x)^{\gamma_2}+\delta_1V(x),
\end{align}
for all $x\in \mathbb R^n\setminus\{0\}$ along the trajectories of \eqref{ex sys} with $\alpha_1, \alpha_2>0$, $\delta_1\in \mathbb R$, $\gamma_1 = 1+\frac{1}{\mu}$, $\gamma_2 = 1-\frac{1}{\mu}$ for some $\mu>1$. Then, there exists a neighborhood $D\subseteq \mathbb R^n$ of the origin such that for all $x(0)\in D$, the trajectories of \eqref{ex sys} satisfy $x(t)\in D$ for all $t\geq 0$, and reach the origin within a fixed time $T$, where{\small
\begin{align}
    D & = \begin{cases} \; \mathbb R^n; &  r < 1,\\
    \left\{x\; |\; V(x)\leq k^\mu\left(\frac{\delta_1-\sqrt{\delta_1^2-4\alpha_1\alpha_2}}{2\alpha_1}\right)^\mu\right\}; & r\geq 1, 
    \end{cases},\label{eq: domain of attraction}\\
    T & \leq \begin{cases}\frac{\mu\pi}{2\sqrt{\alpha_1\alpha_2}};& \hspace{66pt} r\leq 0,\\
    \frac{\mu}{\alpha_1k_1}\left(\frac{\pi}{2}-\tan^{-1}k_2\right); & \hspace{66pt} 0 \leq r<1,\\
    \frac{\mu k}{(1-k)\sqrt{\alpha_1\alpha_2}}; & \hspace{66pt} r \geq 1,
    \end{cases},\label{new FxTS T est}
\end{align}}\normalsize
where $r\coloneqq \frac{\delta_1}{2\sqrt{\alpha_1\alpha_2}}$, $0<k<1$, 
$k_1 = \sqrt{\frac{4\alpha_1\alpha_2-\delta_1^2}{4\alpha_1^2}}$ and $k_2 = -\frac{\delta_1}{\sqrt{4\alpha_1\alpha_2-\delta_1^2}}$.
\end{Theorem}
\begin{proof}
Note that the domain of attraction $D$ and the time of convergence $T$ are functions of the ratio $r\coloneqq \frac{\delta_1}{2\sqrt{\alpha_1\alpha_2}}$.  The three cases, namely, $r \leq 0$, $0\leq r< 1$ and $r\geq 1$ are studied separately. 

For $r\leq 0$, one can recover the right-hand side of \eqref{eq: dot V FxTS old } from \eqref{eq: dot V new ineq}, and it follows from Lemma \ref{FxTS TH} that $D = \mathbb R^n$, and from part (i) in Lemma \ref{lemma:int dot V} with $r \leq 0$, it follows that $T\leq \frac{\mu\pi}{2\sqrt{\alpha_1\alpha_2}}$. 

Next, consider the case when $0\leq r< 1$. First it is shown that there exists $D\subseteq \mathbb R^n$ containing the origin such that $\dot V(x)<0$ for all $x\in D\setminus\{0\}$, so that any sub-level set of the function $V$ contained in $D$ is forward-invariant. The right-hand side of \eqref{eq: dot V new ineq} can be re-arranged so that \eqref{eq: dot V new ineq} reads
\begin{align*}
   \dot V(x)\leq  V(x)\left(-\alpha_1V(x)^{\gamma_1-1}-\alpha_2V(x)^{\gamma_2-1}+\delta_1\right).
\end{align*}
Note that $V(x)>0$ for all $x\neq 0$. Thus, for $\dot V(x)$ to take negative values for all $x\neq 0$, it is needed that 
\begin{align*}
  &  \min_{x\neq 0}\left(-\alpha_1V(x)^{\gamma_1-1}-\alpha_2V(x)^{\gamma_2-1}+\delta_1\right)<0 \\
   \iff & \delta_1<\min_{x\neq 0}\left(\alpha_1V(x)^{\gamma_1-1}+\alpha_2V(x)^{\gamma_2-1}\right)\\
    \iff  & \delta_1<\min_{x\neq 0}\left(\alpha_1V(x)^\frac{1}{\mu}+\alpha_2V(x)^{-\frac{1}{\mu}}\right).
\end{align*}
Substitute $s = V(x)^\frac{1}{\mu}$ to denote $\alpha_1V(x)^\frac{1}{\mu}+\alpha_2V(x)^{-\frac{1}{\mu}} = \alpha_1s+\frac{\alpha_2}{s}$. Then, the function $p:\mathbb R_+\rightarrow \mathbb R$ defined as $$p(s) \coloneqq \alpha_1s+\frac{\alpha_2}{s}$$ is a strictly convex function since $\frac{d^2p}{ds^2} = \frac{2\alpha_2}{s^3}>0$ for all $k>0$ and has a unique minimizer. The derivative of $p$ reads $\frac{dp}{ds} = \alpha_1-\frac{\alpha_2}{s^2}$, which has a unique root in $\mathbb R_+$ at $s = \sqrt{\frac{\alpha_2}{\alpha_1}}$.\footnote{Only the non-negative root is of interest, since $s  = V(x)^\frac{1}{\mu}\geq 0$.} Thus the minimum is attained for $s=\left(\frac{\alpha_2}{\alpha_1}\right)^{\frac{\mu}{2}}$. Define $V^\star \coloneqq \left(\frac{\alpha_2}{\alpha_1}\right)^{\frac{\mu}{2}}$ and $\delta^\star \coloneqq \alpha_1(V^\star)^\frac{1}{\mu}+\alpha_2(V^\star)^{-\frac{1}{\mu}} = 2\sqrt{\alpha_1\alpha_2}$ so that $\alpha_1V(x)^\frac{1}{\mu}+\alpha_2V(x)^{-\frac{1}{\mu}}\geq \delta^\star$ for all $x\in \mathbb R^n$. Thus, for $r<1$, it holds that $\delta_1<\delta^\star\leq \alpha_1V(x)^\frac{1}{\mu}+\alpha_2V(x)^{-\frac{1}{\mu}}$ for all $x$, and so, $\dot V(x)<0$ for all $x\in \mathbb R^n\setminus\{0\}$. Since $D$ is defined as the largest sub-level set of $V$ such that $\dot V(x)$ takes negative values for $x\in D\setminus \{0\}$, it holds that in the case when $0\leq r<1$, $D = \mathbb R^n$. 

\begin{figure}[!ht]
    \centering
        \includegraphics[width=0.9\columnwidth,clip]{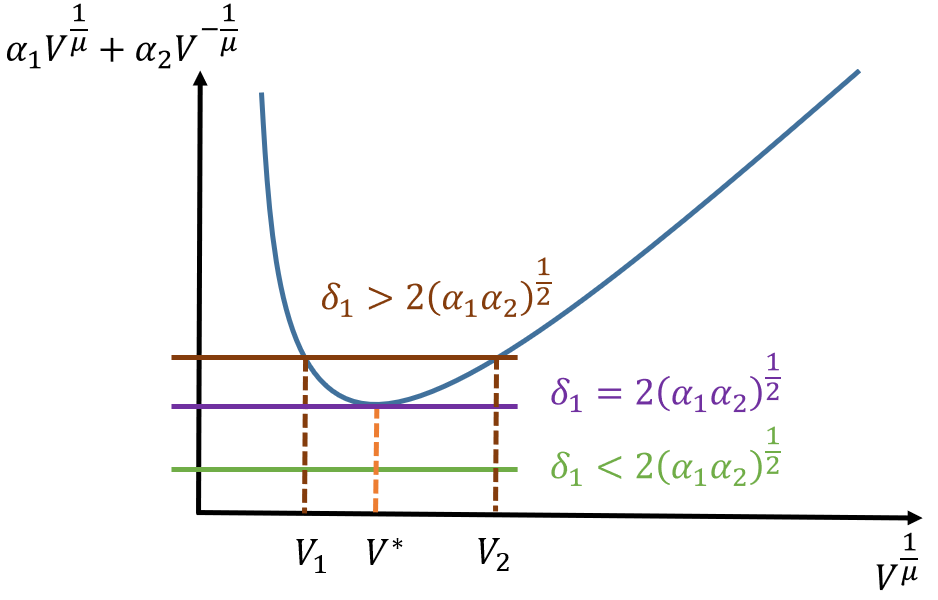}
    \caption{Qualitative variation of $h(V)= \alpha_1V^\frac{1}{\mu}+\alpha_2V^{-\frac{1}{\mu}}$ with $V$, for $\mu>1$.}\label{fig:dot V RHS}
\end{figure}

Finally, for the case when $r\geq 1$, it holds that $\delta_1 = \alpha_1V(x)^{\frac{1}{\mu}}+\alpha_2 V(x)^{-\frac{1}{\mu}}$ for all $x$ such that $V(x)^\frac{1}{\mu} = V_1$ or $V(x)^\frac{1}{\mu} = V_2$, where $V_1$ and $V_2$ are given as 
\begin{align*}
 V_1 \coloneqq \frac{\delta_1-\sqrt{\delta_1^2-4\alpha_1\alpha_2}}{2\alpha_1}, \quad  V_2\coloneqq \frac{\delta_1+\sqrt{\delta_1^2-4\alpha_1\alpha_2}}{2\alpha_1},
\end{align*} 
(see Figure \ref{fig:dot V RHS}). It can be easily verified that if $r\geq 1$, then for all $x$ such that $V(x)^\frac{1}{\mu} < V_1$, it holds that $\delta_1V(x)<\alpha_1V(x)^{\gamma_1}+\alpha_2V(x)^{\gamma_2}$. Thus, $\dot V(x)<0$ for all $x\in D\setminus\{0\}$ with $D = \{x\; |\; V(x)\leq (kV_1)^\mu\}$ for any $0<k<1$. Since $\dot V(x)<0$ for all $x\in D\setminus \{0\}$, it holds that $D$ is forward invariant and thus, is a domain of attraction. 

So far, the domain of attraction $D$ is computed such that starting from any $x(0)\in D$, the system trajectories reach the origin since $\dot V(x)<0$ for all $x\in D\setminus\{0\}$. Next, it is shown that in all the aforementioned cases, the system trajectories reach the origin within a fixed time for all $x(0)\in D$. 

Let $x(0)\in D$, so that $\dot V(x(t))\leq 0$ for all $t\geq 0$ per the analysis above. Thus, from \eqref{eq: dot V new ineq}, it holds that 
\begin{align*}
    & \frac{1}{-\alpha_1V^{\gamma_1}-\alpha_2V^{\gamma_2}+\delta_1V}\frac{dV}{dt} \geq 1, \\
    \implies & \int_{0}^{T}\frac{1}{-\alpha_1V^{\gamma_1}-\alpha_2V^{\gamma_2}+\delta_1V}\frac{dV}{dt}dt = \int_{0}^Tdt,\\
    \implies & \int_{V(x(0))}^{V(x(T))}\frac{dV}{-\alpha_1V^{\gamma_1}-\alpha_2V^{\gamma_2}+\delta_1V}\geq \int_{0}^Tdt = T,\\
    \implies & \int_{V_0}^{0}\frac{dV}{-\alpha_1V^{\gamma_1}-\alpha_2V^{\gamma_2}+\delta_1V}\geq T,
\end{align*}
where $V_0 = V(x(0))$ and $T$ is the time instant when the trajectories reach the origin. Denote the left-hand side of last inequality above as $I$, so that $I\geq T$. 
The cases when $0\leq r<1$ and $r\geq 1$ are considered separately. 

First, let $0\leq r<1$. Using part (i) in Lemma \ref{lemma:int dot V}, it holds that 
\begin{align}\label{eq: T bound case 1}
    T\leq I \overset{\eqref{eq: I bound 1}}{\leq}\frac{\mu}{\alpha_1k_1}\left(\frac{\pi}{2}-\tan^{-1}k_2\right),
\end{align}
where $k_1 = \sqrt{\frac{4\alpha_1\alpha_2-\delta_1^2}{4\alpha_1^2}}$ and $k_2 = -\frac{\sqrt{\delta_1}}{\sqrt{4\alpha_1\alpha_2-\delta_1^2}}$.
Hence, if $\delta_1<2\sqrt{\alpha_1\alpha_2}$, it holds that $\dot V(x(t))<0$ for all $t\geq 0$ and $V(x(t))= 0$ for all $t\geq T$, for all $x(0)\in \mathbb R^n \setminus \{0\}$, where $T\leq \frac{\mu}{\alpha_1k_1}\left(\frac{\pi}{2}-\tan^{-1}k_2\right)$. Since $V$ is proper, the origin is globally FxTS.

Now, for $r\geq 1$, using part (ii) in Lemma \ref{lemma:int dot V}, it holds that
\begin{align}\label{eq: T bound case 2}
    T\leq \frac{\mu k}{(1-k)\sqrt{\alpha_1\alpha_2}}.
\end{align}
The bounds on $T$ in \eqref{eq: T bound case 1} and \eqref{eq: T bound case 2} are independent of the initial condition $x(0)$. Thus, for all $x(0)\in D\setminus\{0\}$, the origin is FxTS.
\end{proof}

Theorem \ref{Th: FxTS new} gives an expression for the domain of attraction $D$ and the time of convergence $T$ as a function of $\alpha_1, \alpha_2, \delta_1$. As thus, Lemma \ref{FxTS TH} and other similar results in the literature (e.g. \cite{parsegov2012nonlinear}) are special cases of Theorem \ref{Th: FxTS new}.

\begin{Remark}
The domain of attraction $D$ in \eqref{eq: domain of attraction} and the time of convergence $T$ in \eqref{new FxTS T est} are functions of the ratio $\frac{\delta_1}{2\sqrt{\alpha_1\alpha_2}}$. In particular, if $\delta_1<2\sqrt{\alpha_1\alpha_2}$, then per \eqref{eq: domain of attraction}, the domain of attraction is the entire $\mathbb R^n$, and for a given $\alpha_1, \alpha_2$, as $\delta_1$ increases, the domain of attraction shrinks. 
\end{Remark}


\vspace{2pt}
\noindent \textbf{Robustness perspective}: In comparison to Lemma \ref{FxTS TH}, Theorem \ref{Th: FxTS new} allows a positive term $\delta_1V$ in the upper bound of the time derivative of the Lyapunov function. This term also captures the robustness against a class of Lipschitz continuous, or vanishing, additive disturbances in the system dynamics, as shown in the following result. Consider the system 
\begin{align}\label{eq:pert sys}
    \dot x = f(x) + \psi(x),
\end{align}
where $f, \psi:\mathbb R^n\rightarrow \mathbb R^n$, $f(0) = 0$ and there exists $L>0$ such that for all $x\in \mathbb R^n$, $\|\psi(x)\|\leq L\|x\|$. 

\begin{Corollary}\label{cor: robust phi FxTS}
Let the origin for the \textit{nominal} system $\dot x = f(x)$ be FxTS and assume that there exists a Lyapunov function $V:\mathbb R^n \rightarrow\mathbb R$ satisfying the conditions of Lemma \ref{FxTS TH}. Assume that there exist $k_1, k_2>0$ such that $V(x) \geq k_1\|x\|^2$ and $ \left\|\frac{\partial V}{\partial x}\right\|\leq k_2\|x\|$ for all $x\in \mathbb R^n$. Then, the origin of the perturbed system \eqref{eq:pert sys} is FxTS. 
\end{Corollary}
\begin{proof}
The time derivative of $V$ along the system trajectories of \eqref{eq:pert sys} reads
\begin{align*}
    \dot V = \frac{\partial V}{\partial x}f(x) + \frac{\partial V}{\partial x}\psi(x)\leq &-aV^p-bV^q + k_2L\|x\|^2\\
    \leq & -aV^p-bV^q + \frac{k_2L}{k_1}V.
\end{align*}
Hence, using Theorem \ref{Th: FxTS new}, we obtain the origin of \eqref{eq:pert sys} is FxTS for all $x(0)\in D$, where $D$ is a neighborhood of the origin. As per the conditions of Theorem \ref{Th: FxTS new}, $D= \mathbb R^n$ or $D \subset \mathbb R^n$, depending upon whether the ratio $\frac{k_2L}{2k_1\sqrt{ab}}$ is less or greater than 1, respectively.
\end{proof}

\section{Control Synthesis under Input Constraints}\label{sec: QP sims}
In this section, we use the Lyapunov condition \eqref{eq: dot V new ineq} in conjunction with Theorem \ref{Th: FxTS new} in a QP formulation to compute a control input so that the closed-loop trajectories reach a desired goal set within a fixed time. Consider the system:
\begin{align}\label{cont aff sys}
    \hspace{-5pt}\dot x(t) = f(x(t)) + g(x(t))u(x(t)),
\end{align}
where $x\in \mathbb R^n$ is the state vector, $f:\mathbb R^n\rightarrow \mathbb R^n$ and $g:\mathbb R^n\rightarrow \mathbb R^{n\times m}$ are continuous functions, and $u\in\mathcal U\subset \mathbb R^m$ is the control input where $\mathcal U$ denotes the control input constraint set. In addition, consider a goal set, to be reached in a user-defined fixed time $T$, defined as $S_G \coloneqq  \{x\; |\; h_G(x)\leq 0\}$, where $h_G:\mathbb R^n\rightarrow\mathbb R$ is a continuously differentiable function. 
Consider the QP
\begin{subequations}\label{QP gen}
\begin{align}
\min_{v\in \mathbb R^m, \delta_1\in \mathbb R} \; \frac{1}{2}z^THz & + F^Tz\\
    \textrm{s.t.} \quad \quad  A_uv  \leq & b_u, \label{C1 cont const}\\
    L_fh_G(x) + L_gh_G(x)v  \leq & \delta_1h_G(x)-\alpha_1\max\{0,h_G(x)\}^{\gamma_1} \nonumber\\
    & -\alpha_2\max\{0,h_G(x)\}^{\gamma_2} \label{C2 stab const},
\end{align}
\end{subequations}
where $z \coloneqq \begin{bmatrix}v^T & \delta_1\end{bmatrix}^T\in \mathbb R^{m+1}$, $H \coloneqq \textrm{diag}\{p_{u_1},\ldots, p_{u_m}, p_1\}$ is a diagonal matrix consisting of positive weights $p_{u_i}, p_1>0$, $F \coloneqq \begin{bmatrix}\mathbf 0_m^T & q_1\end{bmatrix}$ with $q_1>0$ and $\mathbf 0_m\in \mathbb R^m$ a vector consisting of zeros. The parameters $\alpha_1, \alpha_2, \gamma_1, \gamma_2$ are fixed, and are chosen as $\alpha_1 = \alpha_2 = \frac{\mu\pi}{2\bar T}$, $\gamma_1 = 1+\frac{1}{\mu}$ and $\gamma_2 = 1-\frac{1}{\mu}$ with $\mu>1$. Constraint \eqref{C1 cont const} encodes the control input constraints $u\in \mathcal U = \{v\; |\; A_uv\leq b_u\}$, while \eqref{C2 stab const} encodes the FxT-CLF condition. Below, we show that the QP \eqref{QP gen} is feasible, and under certain conditions, the control input defined as the solution of \eqref{QP gen} lead to FxTS convergence of the closed-loop trajectories. Let the solution of \eqref{QP gen} be denoted as $z^*(\cdot) = \begin{bmatrix}v^*(\cdot)^T & \delta_1^*(\cdot)\end{bmatrix}^T$. 

\begin{Lemma}
If the set $\mathcal U$ is non-empty, then the QP \eqref{QP gen} is feasible for all $x\notin S_G$. 
\end{Lemma}
\begin{proof}
Choose any $\bar v\in \mathcal U$, and since $\mathcal U$ is non-empty, there exists such $\bar v$. For $x\notin S_G$, we have that $h_G(x)>0$ by definition, and thus $h_G(x) \neq 0$. Define $$\bar \delta_1= \frac{L_fh_G(x) + L_gh_G(x)\bar v+\alpha_1h_G(x)^{\gamma_1}+\alpha_2h_G(x)^{\gamma_2}}{h_G(x)},$$ so that \eqref{C2 stab const} is satisfied. Thus, the couple $(\bar v, \bar \delta_1)$ satisfies the constraints of the QP \eqref{QP gen} and hence, the QP \eqref{QP gen} is feasible, for all $x\notin S_G$. 
\end{proof}

The feasibility of \eqref{QP gen} is guaranteed because of the presence of the slack term $\delta_1V$. Note that in the absence of such a term, \eqref{QP gen} might be infeasible due to the presence of the control input constraints. This is one of the main advantages of using the condition \eqref{eq: dot V new ineq} as compared to \eqref{eq: dot V FxTS old }. The following result is adapted from \cite{garg2019prescribedTAC} that shows FxTS of the set $S_G$ under the control input $u(x) = v^\star(x)$. 

\begin{Theorem}[\hspace{-0.3pt}\cite{garg2019prescribedTAC}]\label{Th: d1 d2 P1 solve}
The closed-loop trajectories under the effect of the control input defined as $u(x) = v^*(x)$ reach the set $S_G$ within a fixed time $\bar T$ for all $x(0)\in D$, where:
\begin{itemize}
    \item[(i)] $D = \mathbb R^n$ and $\bar T = T$ if $\max\limits_{x}\delta_1^*(x) \leq 0$;
    \item[(ii)] $D = \mathbb R^n$ and $\bar T \leq \sup\limits_{x} \frac{\mu}{\alpha_1k_1(x)}\left(\frac{\pi}{2}-\tan^{-1}k_2(x)\right)$, where $k_1(x) = \sqrt{\frac{4\alpha_1\alpha_2-\delta_1^*(x)^2}{4\alpha_1^2}}$ and $k_2(x) = -\frac{\delta_1^*(x)}{\sqrt{4\alpha_1\alpha_2-\delta_1^*(x)^2}}$ if $\max\limits_{x}\delta_1^*(x) < 2\sqrt{\alpha_1\alpha_2}$;
    \item[(iii)] $D = \{z\; |\; V(z) \leq \inf\limits_x\left(\frac{\delta_1(x)-\sqrt{\delta_1(x)^2-4\alpha_1\alpha_2}}{2\alpha_1}\right)^\mu\}$ and $\bar T \leq \frac{\mu k}{(1-k)\sqrt{\alpha_1\alpha_2}}$ if $\max\limits_{x}\delta_1^*(x)>2\sqrt{\alpha_1\alpha_2}$. 
\end{itemize}
\end{Theorem}

\noindent From the expression of domain of attraction $D$ in \eqref{eq: domain of attraction}, it can be observed that the domain $D$ shrinks as the ratio $r$ increases. In particular, for the case when $\alpha_1 = \alpha_2  = \alpha$ and $r\geq 1$, the domain of attraction is given as $D = \{x\; |\; V(x)\leq \left(r-\sqrt{r^2-1}\right)^\mu\}$ where $r = \frac{\delta_1}{2\alpha}$. Under the same conditions, the domain of attraction for the closed-loop system \eqref{cont aff sys} under $u = v^\star$ is given as $D = \{x\; |\; V(x)\leq \left(r_M-\sqrt{r_M^2-1}\right)^\mu\}$, which is a function of $r_M \coloneqq \sup \frac{\delta_1^\star(x)}{2\alpha}$. 

\begin{figure}[b]
    \centering
        \includegraphics[width=0.9\columnwidth,clip]{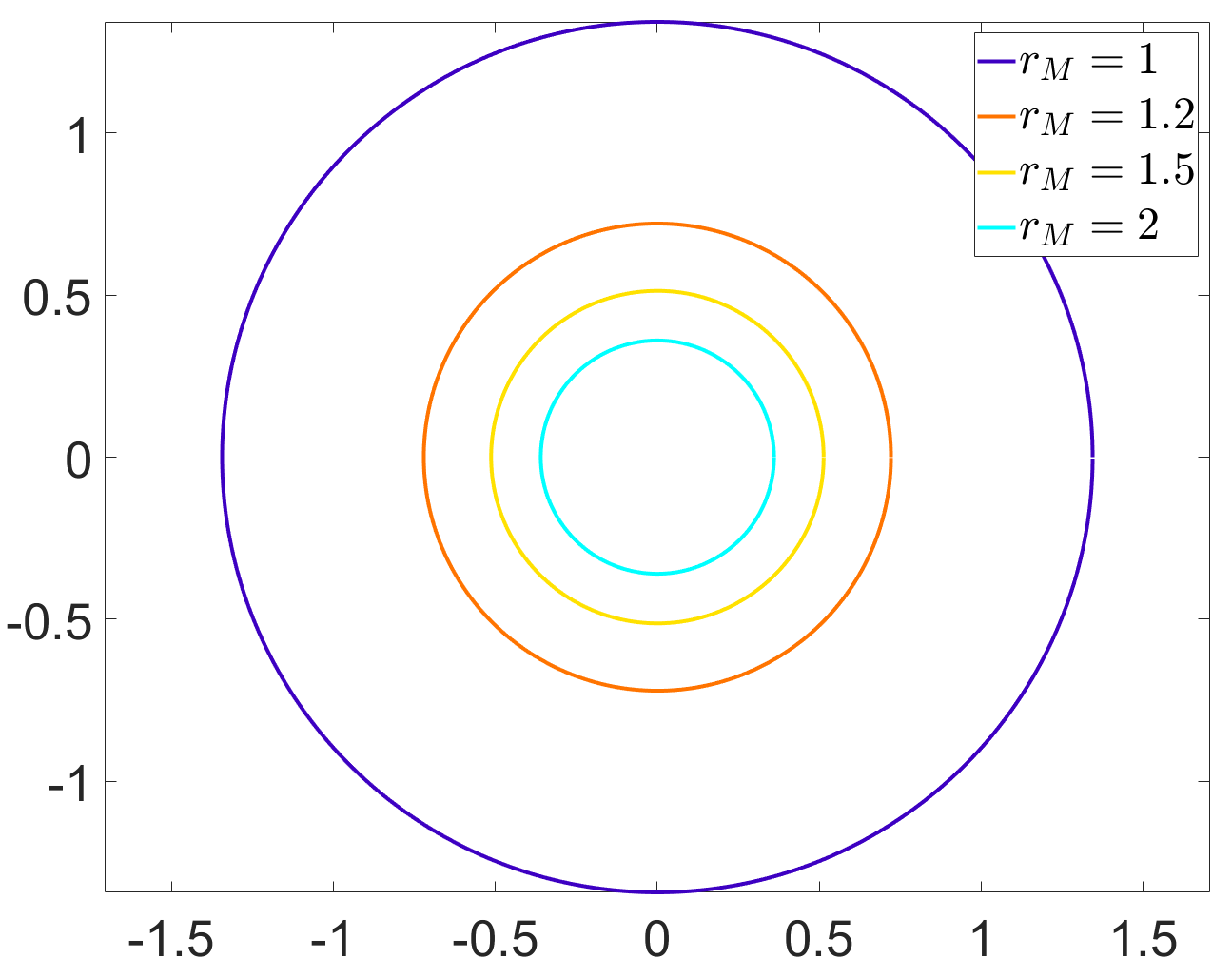}
    \caption{Domain of attraction $D$ for various values of $r_M$.}\label{fig:DoA}
\end{figure}

Figure \ref{fig:DoA} plots the boundary of the set $D$ for various values of $r_M$. Thus, if $\delta_1$ is \textit{small} relative to $\alpha_1, \alpha_2$, then the domain of attraction is large for fixed-time convergence, i.e., the slack term corresponding to $\delta_1$ in QP \eqref{QP gen} characterizes the trade-off between the domain of attraction and time of convergence for given control input bounds. Intuitively, for a given control input constraint set, a larger value of $T$ results into smaller values of $\alpha_1, \alpha_2$, which would result in satisfaction of \eqref{C2 stab const} with smaller value of $\delta_1$. Conversely, for a given $T$ (and thus, for a given pair $\alpha_1, \alpha_2$), a larger control authority would result into satisfaction of \eqref{C2 stab const} with smaller $\delta_1$. This relation between the domain of attraction, the input constraints, and the time of convergence is studied in more detail in \cite{garg2019prescribedTAC}, where KKT conditions are used to compute the closed-form expression for the solution of the QP \eqref{QP gen}, and it is shown that increasing the control bounds or the require time of convergence results in smaller values of $r_M$. 
In this paper, we verify this relation via numerical simulations.

\vspace{2pt}
\noindent \textbf{Numerical experiments}: We consider the following system:
\begin{align*}
    \dot x_1 &= x_2 + x_1(x_1^2+x_2^2-1)+x_1u, \\
    \dot x_2 &= -x_1+\zeta(x_2)(x_1^2+x_2^2-1)+x_2u,
\end{align*}
where $x = [x_1, x_2]^T\in \mathbb R^2, u\in \mathbb R$, $\zeta(z) = (0.8+0.2e^{-100|z|})\tanh(z)$ and $S_G = \{x\; |\; \|x\|\leq 1\}$. Note that in the absence of the control input, the trajectories diverge away from $S_G$, i.e., the set $S_G$ is unstable for the open-loop system. We define $h_G(x) = \|x\|^2-1$. We impose control input bounds of the form $\|u\|\leq u_{max}$, where $u_{max}>0$. The initial conditions are choosen as $x(0) = [3.33, 1.33]^T$.  

\begin{figure}[t]
    \centering
        \includegraphics[width=1\columnwidth,clip]{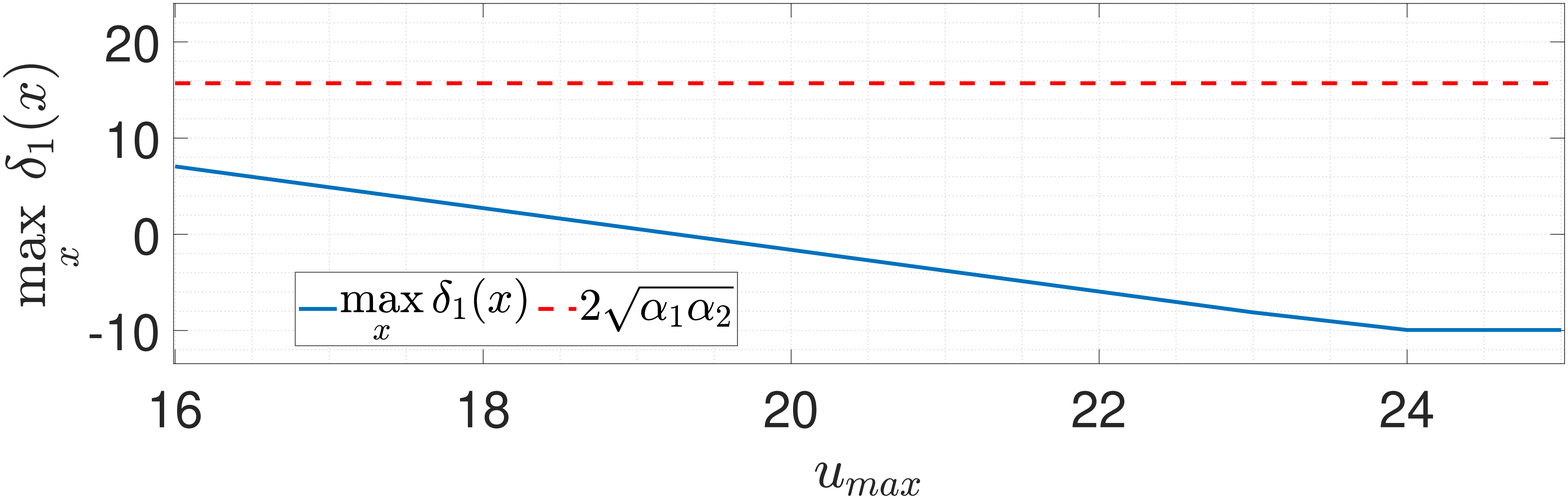}
    \caption{Variation of $\max\delta_1$ for various control input bounds $u_{max}$.}\label{fig:del1 umax}
\end{figure}

We choose $p_{u_1}, p_{u_2} = 1, \mu = 2$ for the numerical simulations. First, we studied the effect of the control input bound on the maximum value of $\delta_1$. We fixed $T = 1, p_1 = 100, q_1 = 1000$, and varied $u_{max}$. Figure \ref{fig:del1 umax} plots the maximum value of $\max_x\delta_1(x)$ for various values of $u_{max}\in [16\;,\; 25]$.\footnote{Since the open-loop system is unstable, for given set of initial conditions, it is observed that the closed-loop trajectories diverge for $u_{max}\leq 16$.} It can be observed that $\delta_1$ decreases as the control authority of the system increases. This verifies the intuition that the domain of attraction expands as the control authority increases. 

\begin{figure}[b]
    \centering
        \includegraphics[width=1\columnwidth,clip]{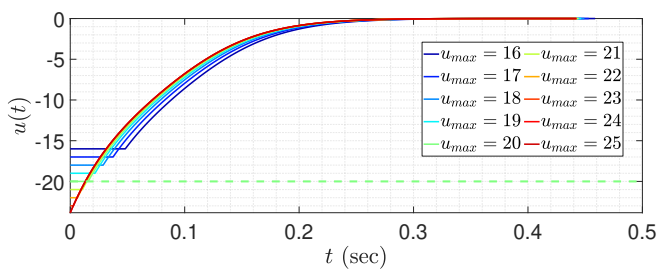}
    \caption{Control input $u(t)$ for various control input bounds $u_{max}$.}\label{fig:u bound umax}
\end{figure}

Figure \ref{fig:u bound umax} plots the norm of the control input with time for various values of $u_{max}$. The value of $u_{max}$ increases from $16$ to $25$ from blue to red. It can be observed that in every case, the system trajectories do utilize the maximum available control authority at the beginning of the simulation, while the control input decreases to zero as the system trajectories approach the goal set. 





Next, we fix $u_{max} = 16, p_1 = 100, q_1 = 1000$ and vary the required time of convergence $T$ between 1 and 10. 
Figure \ref{fig:del1 T} shows the variation of $\max_x\delta_1(x)$ as a function of the convergence time $T$. As $T$ increases (or equivalently, $\alpha_1, \alpha_2$ decrease), the maximum value of $\delta_1(\cdot)$ decreases. This implies that for a larger time of convergence, there is a larger domain of attraction starting from which convergence can be achieved in the given time. 


\begin{figure}[!ht]
    \centering
        \includegraphics[width=1\columnwidth,clip]{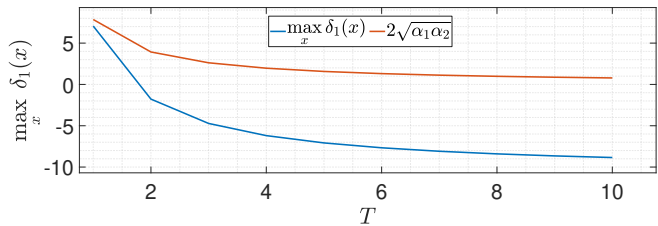}
    \caption{Variation of $\max\delta_1$ for various user-defined convergence time $T$.}\label{fig:del1 T}
\end{figure}



These (numerical) relations indicate that for a required domain of attraction $D$, one can choose the parameters $u_{max}$ and $T$ so that the presented QP in \eqref{QP gen} guarantees FxTS for any initial condition in $D$. Feasibility of the QP \eqref{QP gen} guarantees that for this choice of parameters, a control input exists and renders the goal set FxTS within the chosen time $T$. Conversely, for a given input bound and required time of convergence, it is possible to find the largest domain of attraction by computing the maximum value of $\delta_1$.


\section{Conclusion}\label{sec: conclusion}
We proposed a new result on FxTS by allowing a positive linear term to appear in the time derivative of the Lyapunov function. We characterized the domain of attraction, as well as the upper bound on the time of convergence for fixed-time stability as a function of the coefficients of the positive and the negative terms in the upper bound of the time derivative of the Lyapunov function. We then used the new FxTS result in a QP formulation and showed that the feasibility of the QP is guaranteed due to the presence of the slack term that corresponds to the newly added linear term in our FxTS result. For the QP-based control design technique, we numerically established a relation of the maximum value of this slack term, which characterizes the domain of attraction for fixed-time convergence, with the control input bound, and with the required time of convergence. It is thus shown that with an appropriate choice of the required time of convergence and control input bounds, the presented result can guarantee FxTS from the desired domain of attraction. 

In the future, we would like to study multi-objective problems involving both safety and convergence requirements and find the relations between the largest domain of attraction for fixed-time convergence and the largest subset of the safe set that can be rendered forward invariant, parametrized by the control input bounds and the time of convergence. 

\bibliographystyle{IEEEtran}
\bibliography{myreferences}

\appendices

\section{Proof of Lemma \ref{lemma:int dot V}}\label{app Lemma int dot V proof}
\begin{proof}
We have{\small
\begin{align*}
    I & = \int_{V_0}^{0}\frac{dV}{-\alpha_1V^{\gamma_1}-\alpha_2V^{\gamma_2}+\delta_1V} = \int_{V_0}^{0}\frac{dV}{V(-\alpha_1V^{\frac{1}{\mu}}-\alpha_2V^{\frac{-1}{\mu}}+\delta_1)}.
\end{align*}}\normalsize
Substitute $m = V^{\frac{1}{\mu}}$, so that $dm = \frac{1}{\mu}V^{\frac{1}{\mu}-1}dV$, which implies that $\frac{1}{\mu}\frac{dV}{V} = \frac{dm}{V^\frac{1}{\mu}} = \frac{dm}{m}$. Using this, we obtain that $I =  \mu\int_{V_0^{\frac{1}{\mu}}}^{0}\frac{dm}{(-\alpha_1m^2-\alpha_2+\delta_1m)}$. Now, we consider the three cases, namely, $\delta_1< 2\sqrt{\alpha_1\alpha_2}$, $\delta_1= 2\sqrt{\alpha_1\alpha_2}$ and $\delta_1> 2\sqrt{\alpha_1\alpha_2}$ separately. 

First, consider the cases when $\delta_1< 2\sqrt{\alpha_1\alpha_2}$. In the case, we can re-write $I$ as $I = \mu\int_{V_0^{\frac{1}{\mu}}}^{0}\frac{dm}{-\alpha_1\left((m-\frac{\delta_1}{2\alpha_1})^2+\frac{4\alpha_1\alpha_2-\delta_1^2}{4\alpha_1^2}\right)}$. Evaluating the integral, we obtain
\begin{align*}
    I & = \frac{\mu}{-\alpha_1k_1}(\tan^{-1}k_2-\tan^{-1}k_3),
\end{align*}
where $k_1 = \sqrt{\frac{4\alpha_1\alpha_2-\delta_1^2}{4\alpha_1^2}}$, $k_2 = -\frac{\delta_1}{\sqrt{4\alpha_1\alpha_2-\delta_1^2}}$ and $k_3 = \frac{2\alpha_1V_0^{\frac{1}{\mu}}-\delta_1}{\sqrt{4\alpha_1\alpha_2-\delta_1^2}}$. Using this and $\tan^{-1}(\cdot)\leq \frac{\pi}{2}$, we obtain 
\begin{align*}
    I & = \frac{\mu}{\alpha_1k_1}(\tan^{-1}k_3-\tan^{-1}k_2)\leq \frac{\mu}{\alpha_1k_1}(\frac{\pi}{2}-\tan^{-1}k_2).
\end{align*}

Next, we consider the case when $\delta_1> 2\sqrt{\alpha_1\alpha_2}$. In this case, the roots of $\gamma(m) = 0$ are real. Let $a\leq b$ be the such that $\alpha_1m^2-\delta_1m+\alpha_2 = \alpha_1(m-a)(m-b)$. This substitution allows us to factorize the denominator to evaluate the integral $I$. Note that since $ab = \alpha_2>0$ and $a+b = \delta_1$, we have $0<a\leq b$. Since $V_0^\frac{1}{\mu}\leq k\frac{\delta_1-\sqrt{\delta_1^2-4\alpha_1\alpha_2}}{2\alpha_1} = ka$ where $k<1$, we have that $\frac{1}{-\alpha_1V^{\gamma_1}-\alpha_2V^{\gamma_2}+\delta_1V}<0$ for all $V\leq V_0$, i.e., the denominator $\delta_1V-\alpha_1V^{\gamma_1}+\alpha_2V^{\gamma_2}$ does not vanish for $V\in [0, V_0]$. Thus, we obtain that 
{
\begin{align*}
    I &  = \mu\int_{V_0^{\frac{1}{\mu}}}^{0}\frac{dm}{(-\alpha_1m^2-\alpha_2+\delta_1m)} \\
    & = -\frac{\mu}{\alpha_1}\int_{V_0^{\frac{1}{\mu}}}^{0}\frac{dm}{(m-a)(m-b)}\\
    & = -\frac{\mu}{\alpha_1(a-b)}\left(\int_{V_0^{\frac{1}{\mu}}}^{0}\frac{dm}{m-a}-\int_{V_0^{\frac{1}{\mu}}}^{0}\frac{dm}{m-b}\right).
\end{align*}
}
Evaluating the integrals, we obtain{\small
\begin{align*}
     I & = \frac{-\mu}{\alpha_1(a-b)}\left(\log\left(\frac{a}{|V_0^\frac{1}{\mu}-a|}\right)-\log\left(\frac{b}{|V_0^\frac{1}{\mu}-b|}\right)\right)\\
     & = \frac{\mu}{\alpha_1(a-b)}\left(\log\left(\frac{b}{a}\right)+\log\left(\frac{|V_0^\frac{1}{\mu}-a|}{|V_0^\frac{1}{\mu}-b|}\right)\right)\\
      & \leq \frac{\mu}{\alpha_1(b-a)}\left(\log\left(\frac{b-ka}{a(1-k)}\right)-\log\left(\frac{b}{a}\right)\right),
\end{align*}}\normalsize
It can be easily shown that the above upper-bound on $I$ decreases monotonically as the ratio $\frac{\delta_1}{2\sqrt{\alpha_1\alpha_2}}$ increases from 1 to $\infty$. Thus, the maximum value of this upper-bound is achieved in the limit when $\frac{\delta_1}{2\sqrt{\alpha_1\alpha_2}} = 1$. Note also that $\frac{\delta_1}{2\sqrt{\alpha_1\alpha_2}} \to 1$ implies that $b\to a$ and in the limit, we have $a = b = \frac{\delta_1}{2\alpha_1}$. Thus, the maximum value of the upper-bound can be computed by taking limit $b\to a$ as follows:
\begin{align*}
    I\leq & \frac{\mu}{\alpha_1(b-a)}\left(\log\left(\frac{b-ka}{a(1-k)}\right)-\log\left(\frac{b}{a}\right)\right)\\
    \leq & \sup_{\frac{\delta_1}{2\sqrt{\alpha_1\alpha_2}}\geq 1}\frac{\mu}{\alpha_1(b-a)}\left(\log\left(\frac{b-ka}{a(1-k)}\right)-\log\left(\frac{b}{a}\right)\right)\\
    = & \lim_{\frac{\delta_1}{2\sqrt{\alpha_1\alpha_2}}\to 1}\frac{\mu}{\alpha_1(b-a)}\left(\log\left(\frac{b-ka}{a(1-k)}\right)-\log\left(\frac{b}{a}\right)\right)\\
    = & \lim_{b\to a}\frac{\mu}{\alpha_1(b-a)}\left(\log\left(\frac{b-ka}{a(1-k)}\right)-\log\left(\frac{b}{a}\right)\right)\\
     = & \frac{\mu k}{(1-k)\alpha_1 a} = \frac{\mu k}{(1-k)\sqrt{\alpha_1\alpha_2}}, 
\end{align*}
which completes the proof. 
\end{proof}

\end{document}